\documentclass[12pt]{article}
\usepackage{amsthm}
\usepackage{amsfonts}
\usepackage{a4}
\usepackage{epsfig}
\usepackage{stmaryrd}

\def\RR{{\mathbb R}}
\def\PP{{\mbox{\bf P}}}

\def\cZ{{\cal Z}}

\newtheorem{theorem}{Theorem}
\newtheorem{lemma}[theorem]{Lemma}
\newtheorem{proposition}[theorem]{Proposition}
\newtheorem{remark}[theorem]{Remark}

\def\id{\mathrm{Id}}
\newcommand{\I}[1]{{\mathbb #1}}
\newcommand{\IG}{{\mathbb G}}

\newcommand{\dd}{\hspace{3pt}\mathrm{d}}
\newcommand{\AND}{\,\&\,}
\renewcommand{\mid}{:}

\begin{document}

\title{Quasirandom permutations are characterized\\ by 4-point densities}
\author{Daniel Kr\'al'\thanks{Mathematics Institute, DIMAP and Department of
Computer Science, University of Warwick, Coventry CV4 7AL, United Kingdom.
E-mail: {\tt
D.Kral@warwick.ac.uk}. Previous affiliation: Computer Science Institute, Faculty
of Mathematics and Physics, Charles University, Prague, Czech Republic. Part of
the work leading to this invention has received funding from the European
Research Council under the European Union's Seventh Framework Programme
(FP7/2007-2013)/ERC grant agreement no.~259385.}
\and
        Oleg Pikhurko\thanks{Mathematics Institute and DIMAP, University of
Warwick, Coventry, CV4 7AL, United
Kingdom. E-mail: \texttt{O.Pikhurko@warwick.ac.uk}. This author was supported by  the
European Research Council (grant agreement no.~306493)
and the National Science Foundation of the USA (grant DMS-1100215)}}
\date{}\maketitle
\begin{abstract}
 For permutations $\pi$ and $\tau$ of lengths $|\pi|\le|\tau|$, let 
$t(\pi,\tau)$ be the probability that the restriction of $\tau$ to a random
$|\pi|$-point set is (order) isomorphic to $\pi$.
We show that every sequence $\{\tau_j\}$ of permutations such that
$|\tau_j|\to\infty$ and  $t(\pi,\tau_j)\to 1/4!$ for every $4$-point permutation
$\pi$ is \emph{quasirandom} (that is, $t(\pi,\tau_j)\to 1/|\pi|!$ for
every $\pi$).
This answers a question posed by Graham.
\end{abstract}

\section{Introduction}

Roughly speaking, a combinatorial object is called \emph{quasirandom}
if it has properties that a random object has asymptotically almost surely.
This notion has been defined for various structures such
as tournaments~\cite{bib-tour}, set systems~\cite{bib-sets}, subsets of $\I
Z/n\I
Z$~\cite{bib-Zn},
$k$-uniform hypergraphs~\cite{cg90,bib-gowers,bib-gowers07,ht89,KRS}, 
groups~\cite{bib-groups}, etc.

In particular, quasirandomness has been extensively studied for graphs.
Extending earlier results of R{\"odl}~\cite{bib-rodl86} and
Thomason~\cite{bib-thomason87}, 
Chung, Graham and Wilson~\cite{bib-chung89+} gave seven equivalent properties of graph
sequences such that
the sequence of random graphs $\{G_{n,1/2}\}$ possesses them with probability one.
These properties include densities of subgraphs,
values of eigenvalues of the adjacency matrix or
the typical size of the common neighborhood of two vertices. 
In particular, it follows from the results in~\cite{bib-chung89+}
that if the density of $4$-vertex subgraphs in a large graph
is asymptotically the same as in $G_{n,1/2}$, then this is
true for every fixed subgraph. Graham (see~\cite[Page 141]{bib-cooper04}) asked whether
a similar phenomenon also occurs in the case of permutations.

Let us state his question more precisely. Let $S_k$ consist
of permutations on $[k]:=\{1,\dots,k\}$. We view
each $\pi\in S_k$ as a bijection $\pi:[k]\to [k]$ and call $|\pi|:=k$ its \emph{length}. 
For $\pi\in S_k$ and $\tau\in S_m$ with $k\le m$,
let $t(\pi,\tau)$ be the probability that a random $k$-point subset $X$
of $[m]$ induces a permutation \emph{isomorphic} to $\pi$ (that is,
$\tau(x_i)\le \tau(x_j)$ iff $\pi(i)\le \pi(j)$ where $X$ consists
of $x_1<\dots<x_k$). A sequence $\{\tau_j\}$ of permutations 
has \emph{Property $\PP(k)$}
if $|\tau_j|\to\infty$ and $t(\pi,\tau_j)=1/k!+o(1)$ for every $\pi\in
S_k$. It is easy to see
that $\PP(k+1)$ implies $\PP(k)$.
Graham asked whether there exists an integer $m$ such that $\PP(m)$ implies
$\PP(k)$ for every $k$. Here we answer this question:

\begin{theorem}\label{th:main}
Property $\PP(4)$ implies Property $\PP(k)$ for every $k$.
\end{theorem}

It is trivial to see that $\PP(1)\not\Rightarrow \PP(2)$ and an
example 
that $\PP(2)\not\Rightarrow\PP(3)$ can be found
in~\cite{bib-cooper04}. An unpublished manuscript of Cooper and Petrarka~\cite{bib-cp}
shows that $\PP(3)\not\Rightarrow\PP(4)$ and mentions that Chung could also show this
(as early as 2001). 
Being unaware of~\cite{bib-cp}, we found yet another
example that $\PP(3)\not\Rightarrow\PP(4)$. Since it is quite different from the 
construction in~\cite{bib-cp}, we present it in Section~\ref{sect-counter}.

Since these notions deal with properties of sequences of permutations, 
we find it convenient to operate with an appropriately
defined ``limit object'', analogous to that for graphs
introduced by Lov\'asz and Szegedy~\cite{bib-lovasz+szegedy}.
Here we use the analytic aspects of
permutation limits that were studied by Hoppen et al.~\cite{bib-hkmrs1,bib-hkms2}
and we derive Theorem~\ref{th:main} from
its analytic analog 
(Theorem~\ref{thm-main}).

Let the (normalized) \emph{discrepancy} $d(\tau)$ of $\tau\in S_n$ be the
maximum over intervals $A,B\subseteq[n]$ of 
 $$
 \left| \frac{|A|\,|B|}{n^2} - \frac{|\tau(A)\cap B|}{n}\right|.
 $$ 
 Cooper~\cite{bib-cooper04} calls a permutation sequence $\{\tau_j\}$
\emph{quasirandom} if $|\tau_j|\to\infty$ and $d(\tau_j)\to 0$. He also
gives other equivalent properties (\cite[Theorem~3.1]{bib-cooper04})
and he discusses various applications of ``random-like'' permutations.
Using the results of  \cite{bib-hkmrs1,bib-hkms2}, it is not hard to relate
quasirandomness and Properties $\PP(k)$: 

\begin{proposition}\label{equiv} A sequence $\{\tau_j\}$ of permutations is
quasirandom if and only if it satisfies Property $\PP(k)$ for every $k$.
\end{proposition}

The proof of Proposition~\ref{equiv} can be found in Section~\ref{ProofOfProp}. 
Thus our Theorem~\ref{th:main} implies that $\PP(4)$ alone 
is equivalent to quasi\-randomness. 

Finally, let us remark that McKay, Morse and
Wilf~\cite[Page 121]{bib-mmw} also
defined
a notion of quasirandomness for permutations. Their definition, although
related, is different from that of Cooper as it deals with sequences of
\emph{sets} of permutations.

\section{Limits of permutations}\label{limits}

Here we define convergence of permutation sequences and show how
a convergent sequence can be associated with an analytic limit object.
We refer the reader to \cite{bib-hkmrs1,bib-hkms2} for more details.

Let $\cZ$ consist of probability measures $\mu$ on the Borel
$\sigma$-algebra of $[0,1]^2$ that have \emph{uniform marginals}, that is,
$\mu(A\times [0,1])=\mu([0,1]\times A)=\lambda(A)$
for every Borel set $A\subseteq [0,1]$, where $\lambda$ is the
Lebesgue measure on $[0,1]$. 

Fix some $\mu\in\cZ$. Let $V_i=(X_i,Y_i)$ for $i\in[k]$ 
be independent random variables with $V_i\sim \mu$ (that is, 
each $V_i$ has distribution $\mu$). We view an outcome
$(X_1,Y_1,\dots,X_k,Y_k)$ as an element of $[0,1]^{2k}$.
For permutations
$\pi,\tau\in S_k$,
let $A_{\pi,\tau}\subseteq [0,1]^{2k}$ correspond to the event
that
 $$
 X_i< X_j \mbox{ iff } \pi(i)<
\pi(j)\ \ \ \AND\ \ \ Y_i< Y_j \mbox{ iff } \tau(i)<
\tau(j).
 $$
 (For example, the first statement above is equivalent to
 $X_{\pi^{-1}(1)}<\dots<X_{\pi^{-1}(k)}$.)
 Since each of the vectors $(X_1,\dots,X_k)$ and $(Y_1,\dots,Y_k)$
is uniformly distributed over $[0,1]^k$, the probability of the
\emph{degenerate event}
 \begin{equation}\label{eq:degenerate}
 D_k:=\big\{X_i=X_j \mbox{ or } Y_i=Y_j\mbox{ for some
$i\not=j$}\big\}\subseteq [0,1]^{2k}
 \end{equation}
 is zero. Note that the sets $A_{\pi,\tau}$ for $\pi,\tau\in S_k$
partition $[0,1]^{2k}\setminus D_k$.
If we reorder the indices in an outcome $(V_1,\dots,V_k)\in
[0,1]^{2k}\setminus D_k$ so that
$X_1<\dots<X_k$, then the new relative order on $Y_1,\dots,Y_k\in [0,1]$ defines
a random permutation $\sigma(k,\mu)\in S_k$. In other words, if
we land in $A_{\pi,\tau}$, then we set $\sigma(k,\mu)=\tau\pi^{-1}$.
Let the
\emph{density $t(\pi,\mu)$} of $\pi\in S_k$ be the probability that
$\sigma(k,\mu)= \pi$. Equivalently,
 \begin{equation}\label{eq:density}
 t(\pi,\mu)=\sum_{\rho\in S_k} \mu^k({A_{\rho,\pi\rho}})=k!\,
\mu^k(A_{\tau,\pi\tau}), \quad \mbox{any $\tau\in S_k$},
 \end{equation}
 where the last equality uses the fact that $\mu^k({A_{\rho,\pi\rho}})$ does not
depend on $\rho\in S_k$ (because $V_1,\dots,V_k$ are independent and identically distributed). 

A sequence of permutations $\{\tau_j\}$ is {\em
convergent} if $|\tau_j|\to\infty$ and $\{t(\pi,\tau_j)\}$ converges
for every permutation $\pi$. This is the same definition
of convergence as the one in \cite{bib-hkmrs1,bib-hkms2} except we 
additionally require that $|\tau_j|\to\infty$; cf.\
\cite[Claim~2.4]{bib-hkmrs1}. 

It is easy to show that every sequence
of permutations whose lengths tend to infinity has a convergent subsequence; see
e.g.\
\cite[Lemma~2.11]{bib-hkms2}. Furthermore, for every
convergent sequence $\{\tau_j\}$ there is $\mu\in\cZ$
such that for every permutation $\pi$ we have
 \begin{equation}\label{eq:limit}
 \lim_{j\to\infty} t(\pi,\tau_j)=t(\pi,\mu).
 \end{equation}
 For the reader's convenience, we sketch the proof from \cite{bib-hkmrs1} that $\mu$
exists. For $\pi\in S_k$, let $\mu_{\pi}\in\cZ$ be obtained by
dividing the square $[0,1]^2$
into $k\times k$ equal squares and distributing the mass uniformly
on the squares with indices $(i,\pi(i))$, $i=1,\dots,k$. By Prokhorov's
theorem, $\{\mu_{\tau_j}\}$ has a subsequence that  weakly converges to some
measure $\mu$. We have $\mu\in\cZ$ as this set is closed in the weak topology.
Finally, $\mu$ satisfies (\ref{eq:limit}) because, for any fixed $\pi$, the
function $t(\pi,{-}):\cZ\to\RR$ is continuous in the weak topology
and $t(\pi,\tau_j)=t(\pi,\mu_{\tau_j})+O(1/|\tau_j|)$. 

We remark that Hoppen et al.~\cite{bib-hkmrs1,bib-hkms2} proposed a slightly different limit object:
the regular conditional distribution function of $Y$ with respect to $X$, where
$(X,Y)\sim \mu$. 
Lemma~2.2 and Definition~2.3 in \cite{bib-hkmrs1} show how to switch back and forth between
the two objects.

Now, we are ready to state the analytic version of Theorem~\ref{th:main}. Let us call $\mu\in\cZ$
\emph{$k$-symmetric} if $t(\pi,\mu)=1/k!$ for every $\pi\in S_k$.

\begin{theorem}
\label{thm-main}
Every $4$-symmetric $\mu\in\cZ$ is
the (uniform) Lebesgue measure on $[0,1]^2$. In particular, $\mu$
is $k$-symmetric for every $k$.
\end{theorem}

Let us show how Theorem~\ref{thm-main} implies Theorem~\ref{th:main}.
Suppose on the contrary that some $\{\tau_j\}$ satisfies $\PP(4)$
but not $\PP(k)$.
Fix $\pi\in S_k$ and a subsequence $\{\tau_j'\}$
such that $\lim_{j\to\infty} t(\pi,\tau_j')$ exists and
is not equal to $1/k!$.
Consider now a convergent subsequence $\{\tau_j''\}$ of 
$\{\tau_j'\}$
and
let $\mu\in\cZ$ be its limit. By (\ref{eq:limit}), $\mu$ is 4-symmetric and, 
by Theorem~\ref{thm-main}, $\mu$ is $m$-symmetric for every $m$. But then
$\lim_{j\to\infty} t(\pi,\tau_j'')=t(\pi,\mu)=1/k!$,
which is the desired contradiction.

\section{Proof of Theorem~\ref{thm-main}}
\label{sect-main}

In this section, let $\mu\in\cZ$ be
arbitrary with $t(\pi,\mu)=1/4!$ for every
$\pi\in S_4$. Let $\lambda\in\cZ$ denote the uniform measure
on $[0,1]^2$. Our objective is to show that $\mu=\lambda$. 

Let $V=(X,Y)\sim \mu$ and $v=(x,y)\sim\lambda$ be independent. For brevity, 
let us
abbreviate $\int_{[0,1]^2}$ to $\int$. Define a function $F:[0,1]^2\to
[0,1]$ by
 $$
 F(a,b):=\mu([0,a]\times [0,b])=\int_{V\le (a,b)} \dd V,
 $$
  where
$V\le (a,b)$ means that
$X\le a$ and $Y\le b$. Since $\mu$ has uniform
marginals,
the function $F$ is continuous. 

First, we show that the 4-symmetry of $\mu$ uniquely determines certain
integrals.

\begin{lemma}\label{lm-integrals}
 $$
 \int F(X,Y)^2 \dd V =
\int F(X,Y)XY \dd V = 
\int  F(x,y)^2\dd v=\frac19.
 $$
\end{lemma}
 \begin{proof}
Let
$V_i=(X_i,Y_i)\sim\mu$, for $i=1,2,\dots$, be independent random variables
distributed according to $\mu$.  By Fubini's theorem, we have
 $$
 \int F(X,Y)^2 \dd V = \int \left(\int_{V_2\le
V_1} \dd V_2 \right)\left(\int_{V_3\le
V_1} \dd V_3 \right)  \dd V_1 = \int_A
\dd(V_1,V_2,V_3),
 $$
 where $A=\{(V_1,V_2,V_3)\mid V_2\le V_1\AND V_3\le V_1\}\subseteq [0,1]^6$.
Note that $$
 A\setminus D_3=\bigcup_{\pi,\tau\in S_3\atop \pi(1)=\tau(1)=3} A_{\pi,\tau},
 $$
 where $D_3$ is defined by (\ref{eq:degenerate}) and the union is
 over $\pi,\tau\in S_3$ such that $\pi(1)=\tau(1)=3$.
The $4$-symmetry of $\mu$ and~(\ref{eq:density}) imply that
$\mu^k(A_{\pi,\tau})=(1/k!)^2$
for every $k\le 4$ and $\pi,\tau\in S_k$. Since $\mu^3(D_3)=0$, we have
$\mu^3(A)=4\cdot (1/3!)^2=1/9$, as required.

Likewise,
 \begin{equation}\label{eq:int2}
  \int  F(X,Y)XY \dd V	
  =\int_{B} \dd(V_1,\dots,V_4),
   \end{equation}
 where $B\subseteq [0,1]^8$
corresponds to the event that $V_2\le V_1$, $X_3\le X_1$ and $Y_4\le Y_1$.
One can derive (\ref{eq:int2}) by replacing each factor by an integral
(for example, 
$X$ is replaced by $\int_{X_3\le X} \dd V_3$) and applying Fubini's theorem. 

The integral in the right-hand side of (\ref{eq:int2}) is equal to the $\mu^4$-measure of the
union of $A_{\pi,\tau}$ over some (explicit) set of pairs $\pi,\tau\in
S_4$.
The measure of this set is uniquely  determined 
by the 4-symmetry of $\mu$. Thus the
integral does
not change if we replace $\mu$ by any other 4-symmetric measure. Considering
the uniform measure $\lambda$, we obtain $\int x^2y^2 \dd v=1/9$, as required.

Next, observe that $(X_1,Y_2)$ is uniformly distributed in $[0,1]^2$ 
because $V_1$ and $V_2$ are independent and have uniform marginals. Again, the
value
of
 $$
\int  F(x,y)^2\dd v= \int_{[0,1]^4}  F(X_1,Y_2)^2
\dd(V_1,V_2)=\int_{V_3,V_4\le (X_1,Y_2)} \dd(V_1,\dots,V_4), 
 $$
 does not depend on the choice of $\mu$ and can be easily computed
by taking $\mu=\lambda$.\end{proof}

Since $X$ is uniformly distributed in $[0,1]$, we have $\int X^2\dd V=1/3$.
Also,
 $$ \int F(x,y)xy \dd v = 
\int_{v\ge V}xy\dd (v,V) =
\frac{1}{4}\int
(1-X^2-Y^2+X^2Y^2)\dd V.
 $$

We use the above identities and apply the Cauchy-Schwartz inequality twice
to get the following series of inequalities:
 \begin{eqnarray*}
 \frac{1}{81} &=&
  \left(\int F(X,Y)XY\dd V\right)^2\ \le\ 
\left(\int F(X,Y)^2\dd
V\right)\cdot\left(\int X^2Y^2\dd V\right)\\
  &=&
 \frac19\,\left(4\cdot\int  F(x,y)xy\dd v-\int
(1-X^2-Y^2)\dd V\right)\\
 &=&
\frac{1}{9}\left(4\cdot\int F(x,y)xy\dd v-\frac{1}{3}
\right)\\
 &\le&
 \frac{4}{9}\,\sqrt{\int F(x,y)^2\dd v}\cdot \sqrt{\int
x^2y^2\dd v}-\frac{1}{27}\ =\
  \frac{1}{81}.
  \end{eqnarray*}
Thus we have equality throughout.
However,
the last inequality is equality if and only if $ F(a,b)$ is equal to a
fixed multiple of $ab$ almost everywhere
with respect to the uniform measure $\lambda$. Since $F$ is continuous
and $F(1,1)=1$,
we conclude that $ F(a,b)=ab$ for all $(a,b)\in[0,1]^2$. Thus the
measures $\mu$ and $\lambda$ coincide on all rectangles $[0,a]\times [0,b]$. 
Since these rectangles generate
the Borel $\sigma$-algebra on $[0,1]^2$, we have that $\mu=\lambda$  by the
uniqueness
statement of the Carath\'eodory
Theorem. This proves Theorem~\ref{thm-main}.

\begin{remark}\rm Our proof gives other sufficient conditions for
$\mu=\lambda$. For example, it suffices to require that each of the
three integrals of Lemma~\ref{lm-integrals} is $1/9$. The proof of the
lemma shows that, if desired, these integrals can be expressed as
linear combinations of densities $t(\pi,\mu)$ for $\pi\in S_4$. 
The single identity
$(\int F(x,y)xy\dd v)^2= \frac19 \int F(x,y)^2\dd v$ is also
sufficient for proving that $\mu=\lambda$;
however, if written as a polynomial in terms of permutation densities
(by mimicking the
proof of Lemma~\ref{lm-integrals}), it involves $5$-point permutations.
Our method can give other sufficient conditions in this manner; the choice of
which one to use may depend on the available information about the sequence.
\end{remark} 

\begin{remark}\label{rm:unique}\rm 
Also, the argument of Lemma~\ref{lm-integrals} shows that,
for every polynomial $P(x,y)$ and $\mu\in\cZ$,
the value of $\int P(x,y)\dd \mu(x,y)$ can be expressed
as a linear combination of permutation densities. This 
observation combined with the
Stone-Weierstrass Theorem gives the \emph{uniqueness} of a permutation limit: if $\mu,\mu'\in\cZ$
have the same permutation densities, then $\mu=\mu'$ (cf.\
\cite[Theorem~1.7]{bib-hkmrs1}).\end{remark}

\section{$\PP(3)$ does not imply $\PP(4)$}
\label{sect-counter}

First,
we construct a $3$-symmetric measure $\mu\in\cZ$ which is not 4-symmetric. 
For $a\in [0,1]$, let $M(a)$ be the set of all the points $(x,y)\in
[0,1]^2$
such that
$x+y\in\{1-a/2,1+a/2,a/2,2-a/2\}$ or
$y-x\in\{-a/2,a/2,1-a/2,a/2-1\}$. See Figure~\ref{fig-Z}
for illustrations of this definition.
Define $\mu_a\in\cZ$ for $a\in [0,1]$ to be the permutation limit such
that
the mass is uniformly distributed on $M(a)$.
Because of the symmetries of $\mu_a$ (invariance under
the horizontal and vertical reflections), we have that $t(\pi,\mu_a)=1/6$ for every
$\pi\in S_3$
if and only if $t(\id_3,\mu_a)=1/6$, where $\id_3$ is the
identity 3-point permutation.

\begin{figure}
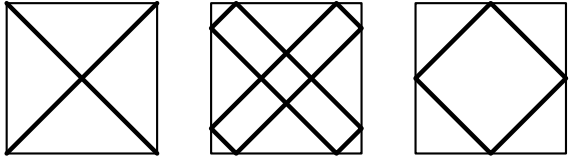

\begin{center}
\epsfbox{qrandperm.101}
\hskip 5mm
\epsfbox{qrandperm.102}
\hskip 5mm
\epsfbox{qrandperm.103}
\end{center}
\caption{The sets $M(0)$, $M(1/3)$ and $M(1)$.}
\label{fig-Z}
\end{figure}

Routine calculations show that $t(\id_3,\mu_0)=1/4$ and
$t(\id_3,\mu_1)=1/8$.
Since $t(\id_3,\mu_a)$ is continuous in $a$,
there exists $b\in [0,1]$ such that $t(\id_3,\mu_{b})=1/6$.
Moreover, $\mu_{b}$ is not 4-symmetric. This can be verified directly;
it also follows from Theorem~\ref{thm-main} since $\mu_{b}$ is not
the uniform measure.

Take a sequence $\{\tau_j\}$ of permutations that converges to
$\mu_{b}$. For example, the random sequence $\{\sigma(j,\mu_{b})\}$
has this property with probability one, see~\cite[Corollary
4.3]{bib-hkms2}. 
Any such sequence $\{\tau_j\}$ satisfies $\PP(3)$ but not $\PP(4)$.

\begin{remark}\label{rm:transform}\rm There are other ways how one can get an example
of a 3-symmetric non-uniform measure by transforming $M(0)$ into $M(1)$.
For example, for $0< a< 1$, let $\nu_a\in\cZ$  assign measure $a$ to $M(0)$
and measure $1-a$ to $M(1)$ with the conditional distributions being
equal to $\mu_0$ and $\mu_1$. Again by 
continuity, there is $a$ such that $\nu_a$ is 3-symmetric.\end{remark}

\begin{remark}\label{rm:inflatable}\rm Let us call a permutation
$\pi\in S_n$ \emph{$k$-inflatable} if $n>1$ and $\mu_\pi$ is $k$-symmetric,
where $\mu_\pi\in\cZ$ is the measure associated with $\pi$ 
as is described after~(\ref{eq:limit}). 
Cooper and Petrarka~\cite{bib-cp} discovered many $3$-inflatable permutations
by computer search,
thus giving examples that $\PP(3)\not\Rightarrow \PP(4)$. The results in~\cite{bib-cp}
show that a shortest
3-inflatable permutation has length $9$ and that $S_9$ has exactly four
$3$-inflatable permutations: $(4,3,8,9,5,1,2,7,6)$, $(4,7,2,9,5,1,8,3,6)$, and
their vertical reflections. Clearly, our Theorem~\ref{thm-main} implies that
no $4$-inflatable permutation can exist. In particular, this proves  (in a stronger form)
Conjecture~3 in~\cite{bib-cp} that no 4-inflatable permutation
with certain properties exists.\end{remark}

\section{Proof of Proposition~\ref{equiv}}\label{ProofOfProp}

Let $\{\tau_j\}$ be an arbitrary sequence of permutations with
$|\tau_j|\to\infty$. Let $\mu_j\in\cZ$ be the measure associated with
$\tau_j$ as is described after (\ref{eq:limit}). 
It is straightforward to verify that
$d(\tau_j)=d(\mu_j)+o(1)$, where 
 $$
 d(\mu):=\sup \big|\lambda(A\times B)-\mu(A\times B)\big|
 $$
 denotes the \emph{discrepancy} of $\mu\in\cZ$,
 with the supremum (in fact, it is maximum) being taken over intervals 
$A,B\subseteq [0,1]$. Also, it is not hard to show (cf.\ Remark~\ref{rm:unique})
that $\{\tau_j\}$ converges to $\mu$ if and only if $\{\mu_j\}$ weakly
converges to $\mu$.

First, suppose that $\{\tau_j\}$ satisfies $\PP(k)$ for each $k$. 
This means that $\{\tau_j\}$ converges to the uniform limit
$\lambda$. For $a,b\in[0,1]$, let $F_j(a,b):=\mu_j([0,a]\times [0,b])$ and $F(a,b):=ab$. Since
$d(\lambda)=0$ and 
$$
 \mu_j([a_1,a_2]\times[b_1,b_2])=F_j(a_2,
b_2)-F_j(a_1,b_2)-F_j(a_2,b_1)+F_j(a_1,b_1),
 $$
 we conclude that $d(\mu_j)\le 4\cdot \|F_j-F\|_\infty$.
The weak convergence $\mu_j\to\lambda$ of measures in $\cZ$ 
gives that $F_j\to
F$
pointwise. Since
$F$ and each function $F_j$, defined on the compact space $[0,1]^2$, are
continuous and monotone in both coordinates, this implies
that
 \begin{equation}\label{eq:Fn}
 \|F_j-F\|_\infty\to 0.
 \end{equation} 
 (Alternatively, (\ref{eq:Fn}) directly follows
from~\cite[Lemma~5.3]{bib-hkmrs1}.)  Thus $d(\mu_j)\to 0$ and $\{\tau_j\}$ is
quasirandom. 

Next suppose that $d(\tau_j)\to 0$. One way to establish
Property~$\PP(k)$ is to use
one of the equivalent definitions of quasirandomness
from~\cite[Theorem~3.1]{bib-cooper04}
(namely Property~[mS]). Alternatively, 
if $\PP(k)$ fails, then (by passing to a
subsequence) we can assume
that $\{\tau_j\}$ converges to some $\mu\in\cZ$ with $\mu\not=\lambda$.
However, we have that $d(\mu)=0$, which implies $\mu=\lambda$,
contradicting our assumption.  This finishes the proof of Proposition~\ref{equiv}.

\section{Concluding remarks}

The theory of flag algebras developed by Razborov~\cite{bib-razborov07}
can be applied to permutation limits: a permutation $\pi:A\to A$
is viewed as two binary relations, each giving a linear order on $A$. For example,
Lemma~\ref{lm-integrals} can be stated and proved within the
flag algebra framework. This view has been helpful for us
when developing our proof.

A graph can be associated with a permutation $\pi\in S_n$ as follows:
let $G(\pi)$ be the graph on $[n]$ with vertices $i<j$ adjacent if
$\pi(i)<\pi(j)$.
Fix $\mu\in\cZ$ and sample a random permutation $\sigma(n,\mu)$.
Define a function $W:[0,1]^4\to\{0,1\}$
by $W(x_1,y_1,x_2,y_2)=1$ if we have 
$(x_1,y_1)<(x_2,y_2)$ or $(x_1,y_1)>(x_2,y_2)$ componentwise and
let $W(x_1,y_1,x_2,y_2)=0$ otherwise.
In other words, $W$ is the indicator function of
the event that $\sigma(2,\mu)$ is the identity 2-point permutation. 
Clearly, $G(\sigma(n,\mu))$ can be generated by sampling independently points
$V_1,\dots,V_n\in [0,1]^2$, each with distribution $\mu$, and connecting those
$i,j\in [n]$
for which $W(V_i,V_j)=1$. The latter procedure corresponds to generating
a random sample $\IG(n,W)$, where $W:[0,1]^2\times [0,1]^2\to[0,1]$ is viewed as a graphon represented
on Borel subsets of $[0,1]^2$ with measure $\mu$,
see~\cite[Section~2.6]{bib-lovasz+szegedy} for details.

Lov\'asz and S\'os~\cite{bib-force1} and Lov\'asz and Szegedy~\cite{bib-force2}
presented various sufficient conditions for a graphon $W$ to be \emph{finitely
forcible} which, in the above notation, means that there is $m$
such that the distribution of $\IG(m,W)$ uniquely determines that of $\IG(k,W)$
for every $k$. As far as we can see, none of these conditions 
directly applies to the graphon associated with the uniform measure
$\lambda\in\cZ$.
Since we answered Graham's question on quasirandom permutations by other means,
we did not pursue this approach any further.

We also refer the reader to Hoppen et al.~\cite[Section~5.3]{bib-hkms1} who
discuss finite forcibility for permutation limits, being motivated by some
questions in
parameter testing.

\section*{Acknowledgments}

The authors thank Carlos Hoppen, Andr\'as M\'ath\'e, and the anonymous referee for helpful comments.

\end{document}